\documentclass[notitlepage]{article}
\usepackage{amsfonts,bm,amsmath,amssymb,amsthm}
\usepackage{cmap}
\usepackage{authblk}
\usepackage[english]{babel}
\usepackage[utf8]{inputenc}
\usepackage{epsfig,color, graphicx}
\usepackage{subfig}
\usepackage[colorlinks=true,urlcolor=blue,citecolor=blue,linkcolor=blue,bookmarks=true]{hyperref}
\usepackage{todonotes}
\newtheorem{theorem}{Theorem}
\newtheorem{lemma}{Lemma}
\newtheorem*{conjecture*}{Conjecture}

\newtheorem{corollary}{Corollary}

\newtheorem{proposition}{Proposition}

\newtheorem*{example*}{Example}
\theoremstyle{definition}
\newtheorem{remark}{Remark}
\DeclareMathOperator*{\argmax}{argmax}
\newcommand{\boldpi}{\boldsymbol{\pi}}
\renewcommand{\le}{\leqslant}
\renewcommand{\ge}{\geqslant}
\title{Bachet's game with lottery moves}

\author{Dmitry Dagaev\footnote{HSE University. Address: 101000, Russia, Moscow, Myasnitskaya street, 20. E-mail: ddagaev@gmail.com. Corresponding author.} ~and Ilya Schurov\footnote{HSE University. Address: 101000, Russia, Moscow, Myasnitskaya street, 20. E-mail: ilya@schurov.com.}}

\date{}
\begin{document}
    \maketitle
    \begin{abstract}
Bachet's game is a variant of the game of Nim. There are $n$ objects in one pile. Two players take turns to remove any positive number of objects not exceeding some fixed number $m$. The player who takes the last object loses. We consider a variant of Bachet's game in which each move is a lottery over set $\{1,2,\ldots, m\}$. The outcome of a lottery is the number of objects that player takes from the pile. We show that under some nondegenericity assumptions on the set of available lotteries the probability that the first player wins in subgame perfect Nash equilibrium converges to $1/2$ as $n$ tends to infinity.

\bigskip

\noindent \textbf{Keywords}: game theory; Bachet's game; backward induction; lotteries.  
    \end{abstract}

    \section{Introduction and main result}
    Bachet's game was formulated in~\cite{Bachet} as follows. Starting from 1, two players add one after another some integer number not exceeding 10 to the sum. The player who is the first to reach 100, wins. This game can be considered as a variant of the game of Nim \cite{Bouton} (other variants can be found, for example, in \cite{Boros2018, Boros2019, Gray, Li, Moore}). One can easily find subgame perfect Nash equilibrium (SPNE) in Bachet's game with backward induction~\cite{Bachet}. 
    
Now assume that at every move instead of choosing the exact number not exceeding some $m$, the player chooses some lottery (i.e. probability distribution)
over numbers $\{1, 2, \ldots, m\}$ from some set of available lotteries, observes realization of the lottery and then makes the corresponding move. Below we provide formal rules of the game that is considered in this paper.

\textbf{Bachet's game with lottery moves} (BGLM). The game is defined by the natural number $n$ of objects in the pile, the natural number $m$ and a set of available lotteries $K\subset S_m$, where $S_m$ is a simplex of all lotteries over numbers $\{1, 2, \ldots, m\}$. Two players take turns to choose a lottery from the set $K$. After making the choice, the player observes realization of the lottery and then takes the corresponding number of objects from the pile. The player who takes the last object loses, including the case when they have to take more objects than remains in the pile. Both players want to maximize the probability of their own victory.

Our main result is the following theorem.

    \begin{theorem}\label{thm:main}
        Fix arbitrary integer $m>1$ and some compact
        set $K\subset S_m$ with the following properties:
        \begin{equation}
            \eta := \max_{(\pi_1,\ldots ,\pi_m) \in K } \max_{i \in \{1, \ldots, m\}} \pi_i < 1;
            \label{eta}
        \end{equation}
        \begin{equation}    
            \nu := \min_{i\in \{1, \ldots, m\}} \max_{(\pi_1,\ldots ,\pi_m) \in K} \pi_i>0.
            \label{nu}        
        \end{equation}
        For any initial number of objects $n$, consider BGLM
        with parameters $n$, $m$, $K$. This game has a non-empty set of SPNE.
        Denote by $p_n$ the probability that the first player wins in arbitrary
        SPNE. 

        Then $p_n$ does not depend on the choice of
        SPNE and
        \begin{equation}\label{eq:main}
            \lim_{n\to\infty}p_n=\frac{1}{2}.
        \end{equation}
    \end{theorem}
    \begin{remark}
        It can be easily proved that if limit~\eqref{eq:main} exists, it has
        to be equal to~$\frac12$. Assume by contradiction that
        limit~\eqref{eq:main} exists and equals $a\neq\frac{1}{2}$. Take some
        $\varepsilon<|a-\frac{1}{2}|$. Then, for some $N$ and all $i\geqslant
        1$, it is true that $|p_{N+i}-a|<\varepsilon$. Consider two cases. If
        $a>\frac12$, then  $p_{N+i}>\frac{1}{2}$ for all $i=1,\ldots,m$ and it
        follows that $p_{N+m+1}<\frac{1}{2}$. Indeed, if any move from the
        initial position leads to a state with winning probability greater than
        $\frac12$, then the winning probability for the initial position is less
        than $\frac12$; formally it follows from~\eqref{eq:pk} below. Similarly,
        if $a<\frac12$, then $p_{N+i}<\frac{1}{2}$ for all $i=1,\ldots,m$ and it
        follows that $p_{N+m+1}>\frac{1}{2}$. This leads us to a contradiction
        with the definition of $N$. Hence, the interesting part is the existence
        of this limit. 
    \end{remark}
    \begin{remark}
        \autoref{thm:main} allows the following interpretation. Assume that the players play classical Bachet's game, but after choosing their move, they make mistakes and play other moves (including suboptimal ones) with some positive probability. Condition \eqref{eta} says that mistakes are unavoidable: there are no pure (i.e., not mixed) moves in the set of all possible moves $K$. This condition is an essential characteristic of BGLM; \eqref{eta} does not hold for classical Bachet's game ($\eta=1$ for the latter). It follows from~\autoref{thm:main} that the presence of unavoidable mistakes drastically changes the outcome of the game for large $n$. Condition~\eqref{nu} says that it is possible to take $i$ objects from the pile, $i=1,\ldots, m$, with positive probability. Condition \eqref{nu} also holds for classical Bachet's game since $\nu=1$ (one can take any number of objects with probability 1). 
    \end{remark}
  
    \begin{conjecture*}
        Though condition \eqref{nu} plays an important technical role in our proof, we believe~\autoref{thm:main} holds true even if this condition is violated.
    \end{conjecture*}

    \begin{remark}
In order to refine the set of all Nash equilibria in games in extensive form,
Selten introduced the concept of the trembling hand \cite{Selten}. This concept
takes into account the lack of perfect rationality and possibility of random
mistakes. If $\Gamma$ is a game in extensive form, construct a perturbed game by
assuming that in each information set of $\Gamma$ a player must mix all
available moves (including suboptimal ones) with some positive weight not less
than the predetermined value (which is a parameter of a particular information
set in a particular perturbed game). Thus, the set of all admissible mixed moves
in a particular information set is a compact subset of the simplex of all
lotteries over pure moves in this information set. This is similar to set $K$ in
the definition of BGLM. The difference is that in BGLM the set of admissible
mixed moves is the same in all information sets. Another major difference is
that in the trembling hand equilibrium concept, the key object is the limit of
the sequence of perturbed games as the severity of random mistakes tends to 0.
We keep this severity parameter ($1-\eta$ in our notation) fixed and vary the
number of objects in the pile, considering infinite horizon limit. Therefore, we
get different perspective on the role of slight mistakes.
     \end{remark}

    \section{Proof of the main result}
    \subsection{Existence of SPNE}
    We find SPNE by backward induction. Fix $m$ and $K$. Obviously, for
    $n=1$, any move leads to losing, as the player has to
    take at least one object in any case. Therefore, any move
    of the first player is in the set of all SPNE and $p_1=0$.

    For convenience reasons, let $p_s=1$ for any $s\le 0$.

    Now assume we proved the existence of SPNE for all BGLM with no more than
    $n=k-1$ objects. Consider BGLM with $n=k$ objects. Assume that after the
    move of the first player, $i$ objects are taken from the pile. The second
    player now plays BGLM with $n=k-i$ objects (becoming `first player' in this
    subgame) and wins it with probability $p_{k-i}$ by induction hypothesis. If
    the second player wins, the first player loses. Therefore, the probability
    that the first player wins in this case is $1-p_{k-i}$.
    By the law of total probility, for move $\boldpi=(\pi_1, \ldots, \pi_m) \in
    K$, the probability
    that the first player wins is given by:
    \begin{equation}\label{eq:wtpk}
        \widetilde{p}_k(\boldpi)=1-\sum_{i=1}^m \pi_i p_{k-i}.
    \end{equation}
    The player wants to maximize this probability by choosing optimal $\boldpi$. Function
    $\widetilde{p}_k$ is continuous with respect to
    $\boldpi$ and therefore attains its maximum value on
    compact set $K$. Then 
    \begin{equation}\label{eq:pk}
        p_k=\max_{\boldpi \in K}\widetilde{p}_k(\boldpi)
    \end{equation}
and $\argmax_{\boldpi} \widetilde{p}_k(\boldpi)$ is non-empty. Obviously, $p_k$
    does not depend on the choice of the move. After the move, the number of objects in the pile will be reduced, hence, the
    existence of SPNE now follows from the induction hypothesis.

    \subsection{Limit behaviour}
    In this section we prove~\eqref{eq:main}. 

    \subsubsection{The notation and the idea of the proof}
    First, we introduce some notation. Let
    \begin{gather*}
        \mathcal{D}_n:=p_n-\frac{1}{2},\quad \Delta_n:=|\mathcal{D}_n|,\\
        W_k=\{k, k-1, \ldots, k-m + 1\},\quad \overline{\Delta}_k=\max_{j \in
            W_k} \Delta_j.
    \end{gather*}
    It is easy to show that sequence $\{\overline{\Delta}_k\}$ is non-increasing
    (see~\autoref{lem:monot} and~\autoref{cor:monot}). Our goal is to show that
    it is strictly decreasing and has zero limit.

    Consider the state of a game with $k+1$ objects in the pile. Due to \eqref{eq:wtpk}-\eqref{eq:pk}, $\mathcal
    D_{k+1}$ is a convex combination of values~$\mathcal D_j$, $j\in W_k$, taken
    with a negative sign. If some of these values taken with nontrivial weights are less
    by absolute value than their maximum possible value $\overline{\Delta}_k$,
    their convex combination is also less than $\overline{\Delta}_k$ by absolute
    value and $\Delta_{k+1} < \overline{\Delta}_k$. Moreover, the gap can be
    estimated from below. This suggests a way to prove that
   sequence $\{\overline{\Delta}_k\}$ is strictly decreasing and tends to zero.

    However, it is also possible that the convex combination for $\mathcal
    D_{k+1}$ includes (with nontrivial weights) only those $\mathcal D_j$
    whose absolute values are (almost) equal to $\overline{\Delta}_k$. In this
    case, $\Delta_{k+1}\approx \overline{\Delta}_{k}$ and no significant drop
    occurs. Such cases should be considered separately.
    
    Due to condition~\eqref{nu}, the player is allowed to put nontrivial
    weight on any move $j$. Due to rationality, the player tends to put larger
    weights on moves with smaller $\mathcal D_j$. The `worst case'
    scenario
    is when all $\mathcal D_j$'s, $j\in W_k$, are positive and (almost) equal to
    $\overline{\Delta}_k$. We show that in this case $\mathcal D_{k-m}$ should
    be negative and significantly larger by absolute value than
    $\overline{\Delta}_k$, see details in \autoref{lem:k-m}.
    This gives us a drop between $\overline{\Delta}_{k-m}$ and $\Delta_{k+1}$.

    Another case that needs special attention is when there are several negative
    values of $\mathcal D_j \approx -\overline{\Delta}_j$, $j\in W_k$. This case
    is covered by \autoref{lem:drop}.  There we prove that significant drops in
    $\Delta_k$ occur at least for every additional $3m$ objects in the pile, and
    the sequence $\{\Delta_k\}$ can be estimated from above by a decreasing
    geometric progression and obtain the main result.

    \subsubsection{Preliminary considerations}
    \begin{lemma}[Monotonicity lemma]
        For every integer $k>1$, $\Delta_{k} \le \overline{\Delta}_{k-1}$.
    \label{lem:monot}    
    \end{lemma}
    \begin{proof}
        It follows from~\eqref{eq:wtpk}-\eqref{eq:pk} that
        \[
            p_{k}=1-\sum_{i=1}^{m}\pi_i p_{k-i}.
        \]
        for some $\boldpi \in S$. We have:
        \begin{multline}
            \Delta_k=|\mathcal D_k|=\left|p_k-\frac{1}{2}\right|=
            \left|\frac{1}{2}-\sum_{i=1}^{m}\pi_i p_{k-i}\right|=
            \left|\sum_{i=1}^{m}\pi_i\left(\frac{1}{2}-p_{k-i}\right)\right|\le\\
            \sum_{i=1}^m
            \pi_i\left|\frac{1}{2}-p_{k-i}\right|=\sum_{i=1}^m \pi_i
            \Delta_{k-i} \le \sum_{i=1}^m \pi_i
            \overline{\Delta}_{k-1}=\overline{\Delta}_{k-1}.
        \end{multline}
    \end{proof}
    \begin{corollary}\label{cor:monot}
        For every integer $k>1$, $\overline{\Delta}_k \le
        \overline{\Delta}_{k-1}$.
    \end{corollary}
    \begin{proof}
        Indeed,
        \begin{multline}
            \overline{\Delta}_k=\max\{\Delta_k, \Delta_{k-1}, \ldots,
                \Delta_{k-m+1}\} \le 
            \max\{\overline{\Delta}_{k-1}, \Delta_{k-1}, \ldots, \Delta_{k-m+1}\}= \\
            \max\{\max\{\Delta_{k-1},\ldots, \Delta_{k-m}\},\Delta_{k-1},\ldots,
                \Delta_{k-m+1}\}=\\
            \max\{\Delta_{k-1},\ldots, \Delta_{k-m}\}=\overline{\Delta}_{k-1}.
        \end{multline}
    \end{proof}
    \begin{lemma}[No long winning series]\label{lem:before-and-after}
        Assume that for some integer $k>m$ and for all $j\in W_k$, $p_j>
        \frac{1}{2}$. Then 
        \begin{equation}
            \label{eq:after}
        p_{k+1}<\frac{1}{2}
        \end{equation}
        and 
        \begin{equation}
            \label{eq:before}
            p_{k-m}\le \frac{1}{2}.
        \end{equation}
    \end{lemma}

    \begin{proof}
        First, let us prove~\eqref{eq:after}. For some $\boldpi \in K$,
        $$
        p_{k+1}=1-\sum_{i=1}^m \pi_i p_{k-i+1} <1-\sum_{i=1}^m \pi_i
        \frac{1}{2}=1-\frac{1}{2}=\frac{1}{2}.
        $$
        Now prove~\eqref{eq:before} by contradiction. Assume
        $p_{k-m}>\frac{1}{2}$. Then one can apply~\eqref{eq:after} with $k$
        decreased by 1
        and prove that $p_{k}$ has to be less than $\frac{1}{2}$.
        Contradiction.
    \end{proof}

    \subsubsection{Worst case analysis}
    \begin{lemma}\label{lem:k-m}
        Assume that for some $\varkappa\in (0,1)$, for some integer $k>1$ and for all $j\in W_k$ the following inequality holds:
        \begin{equation}
        p_{j}\geqslant \frac{1}{2} + (1-\varkappa)\Delta_{k+1}.
        \label{assumption}
        \end{equation}        
                Then the following inequality holds:
        \begin{equation}
        \Delta_{k+1} \leqslant \frac{\eta}{(2-\eta)(1-\varkappa)}\Delta_{k-m}.
        \label{theta}
        \end{equation}
    \end{lemma}    
        
    \begin{proof}
        Consider strategy $\boldpi=(\pi_1, \ldots, \pi_m) \in K$ that allows the
        player facing $k$ objects to reach the winning probability of $p_k$. It
        follows from the definition (see~\eqref{eq:wtpk}) that
        \begin{equation}\label{eq:pk1}
            p_k = 1 - \sum_{i=1}^m p_{k-i} \pi_i.
        \end{equation}
        Note that due to~\autoref{lem:before-and-after}, $p_{k-m}\le
        \frac{1}{2}$ and therefore $p_{k-m}=\frac{1}{2}-\Delta_{k-m}$. Put it
        into~\eqref{eq:pk1}:
        \begin{multline}
        \label{pik}
        p_k  =  1-\left(\pi_m\left(\frac{1}{2}-\Delta_{k-m}\right) +
            \sum\limits_{i=1}^{m-1}p_{k-i}\pi_i\right)=\\
             1-\frac{\pi_m}{2}-\sum_{i=1}^{m-1}p_{k-i}\pi_i+\pi_m \Delta_{k-m}.
        \end{multline}
        Therefore, 
        \begin{equation}\label{eq:pmdkm}
            \pi_m\Delta_{k-m} = p_k - 1 +\frac{\pi_m}{2} +
            \sum\limits_{i=1}^{m-1}p_{k-i}\pi_i.
        \end{equation}
        Estimate $p_k$ and $p_{k-i}$ in~\eqref{eq:pmdkm} from below with
        $\frac12 + (1-\varkappa)\Delta_{k+1}$ using
        lemma assumption~\eqref{assumption}:
        \begin{equation}\label{eq:pmdkm-le}
            \pi_m\Delta_{k-m} \geqslant \frac{1}{2}+(1-\varkappa)\Delta_{k+1}- 1 +\frac{\pi_m}{2} + (1-\pi_m)\left(\frac{1}{2} + (1-\varkappa)\Delta_{k+1}\right).
        \end{equation}
        Here we also used the relation $\sum_{i=1}^{m-1} =1-\pi_m$. Simplifying the right-hand side of inequality~\eqref{eq:pmdkm-le}, we
        get:
        $$
        \pi_m\Delta_{k-m} \geqslant \Delta_{k+1} (1-\varkappa)(2-\pi_m),
        $$
        or
        \begin{equation}
        \Delta_{k-m}\geqslant (1-\varkappa)\frac{2-\pi_m}{\pi_m}\Delta_{k+1}\geqslant (1-\varkappa)\frac{2-\eta}{\eta}\Delta_{k+1}
        \label{seventeen}
        \end{equation}
        (from definition of $\eta$ and Theorem assumption (see \eqref{eta}), it follows that $\pi_m\leqslant\eta<1$). Then \eqref{theta} follows from \eqref{seventeen}. 
    \end{proof}
    \subsubsection{Drop down for losing positions} 
    In this part we show that for every \emph{losing position} (i.e. position
    with winning probability less than $1/2$), there is a `drop down' in the
    value of $\Delta_k$.
    \begin{lemma}\label{lem:drop-down}
        There exists $\delta<1$ such that the following holds: if $p_{k+1}< 1/2$ for
        some $k$, then
        \begin{equation}\label{eq:drop-down}
            \Delta_{k+1} \leqslant \delta \overline{\Delta}_{k-m}.
        \end{equation}
    \end{lemma}  
    We need the following lemma for the proof.
    \begin{lemma}[Corridor lemma]\label{lem:corridor}
        Assume that $p_{k+1}<1/2$. 
        Then
        \begin{equation}
            \label{corridor}
            \max_{\substack{i\in W_k}}
            \left(p_{i}-\left(\frac{1}{2}+\Delta_{k+1}\right)\right)
            \ge \frac{\nu}{1-\nu}\max_{\substack{i\in
                    W_k}}\left(\frac{1}{2}+\Delta_{k+1}-p_{i}\right).
        \end{equation}
    \end{lemma}
    The proof of \autoref{lem:corridor} is rather technical and is relegated to Appendix.
    \begin{proof}[Proof of \autoref{lem:drop-down}]
        Fix arbitrary $\tau$ such that
        \begin{equation}\label{eq:tau-def}
        0 < \tau < \frac{\nu}{1-\nu}\frac{2-2\eta}{2-\eta}.
        \end{equation}
        Such $\tau$ exists since $\nu \in (0,1)$ and $\eta \in (0,1)$. We show that
        $$
            \delta:=\max\left\{\frac{\eta}{2-\eta}\frac{\nu}{\nu-\tau+\nu\tau},
                \frac{1}{1+\tau}\right\}
        $$
        satisfies~\eqref{eq:drop-down}. Due to~\eqref{eq:tau-def}, $0<\delta<1$.

        Consider separately two cases.

        \paragraph{Case 1.} For all $j\in W_k$
        \begin{equation}
        p_j-\frac{1}{2}\leqslant (1+\tau)\Delta_{k+1}.
        \end{equation}
        This inequality can be rewritten as
        \begin{equation}
        p_j-\left(\frac{1}{2}+\Delta_{k+1}\right)\leqslant \tau\Delta_{k+1}.
        \end{equation}
        Since the latter inequality is true for any $j\in W_k$, we obtain:
        \begin{equation}
        \max_{\substack{j\in W_k}} \left(p_{j}-\left(\frac{1}{2}+\Delta_{k+1}\right)\right)\leqslant \tau\Delta_{k+1}.
        \label{upper}
        \end{equation}
        According to Corridor lemma \ref{lem:corridor},
        \begin{equation}
            \max_{\substack{j\in W_k}} \left(p_{j}-\left(\frac{1}{2}+\Delta_{k+1}\right)\right)\geqslant \frac{\nu}{1-\nu} \max_{\substack{j \in W_k}} \left(\frac{1}{2}+\Delta_{k+1}-p_{j}\right).
        \label{lower}
        \end{equation}
        From \eqref{upper} and \eqref{lower} it follows that
        \begin{equation}
        \max_{\substack{j \in W_k}} \left(\frac{1}{2}+\Delta_{k+1}-p_{j}\right)\leqslant \frac{1-\nu}{\nu}\tau\Delta_{k+1}.
        \end{equation}
        Hence, for any $j\in W_k$ it is true that
        \begin{equation}
        \frac{1}{2}+\Delta_{k+1}-p_j\leqslant \frac{1-\nu}{\nu}\tau\Delta_{k+1},
        \end{equation}
        or
        \begin{equation}
        p_j\geqslant \frac{1}{2} + \left(1-\frac{1-\nu}{\nu}\tau\right)\Delta_{k+1}.
        \end{equation}
        Applying \autoref{lem:k-m} with $\varkappa=\frac{1-\nu}{\nu}\tau$, we obtain that
        \begin{equation}
        \Delta_{k+1}\leqslant \frac{\eta}{(2-\eta)\left(1-\frac{1-\nu}{\nu}\tau\right)}\Delta_{k-m},
        \end{equation}  
        or
        \begin{equation}
            \Delta_{k+1}\leqslant
            \frac{\eta}{2-\eta}\frac{\nu}{\nu-\tau+\nu\tau}\Delta_{k-m}\leqslant \delta
            \Delta_{k-m}\le \delta \overline{\Delta}_{k-m}.
        \end{equation}

        \paragraph{Case 2.} There exists $i\in W_k$ such that
        \begin{equation}
        p_i-\frac{1}{2}> (1+\tau)\Delta_{k+1}.
        \end{equation}
        Then, 
        \begin{equation}
        \Delta_{k+1}<\frac{1}{1+\tau}\left(p_i-\frac{1}{2}\right) \leqslant \delta
        \Delta_i\leqslant \delta \overline{\Delta}_i \leqslant \delta
        \overline{\Delta}_{k-m}.
        \end{equation}
        The last inequality is due to \autoref{cor:monot} and the fact that $i>k-m$.
    \end{proof}  
  
    \subsubsection{Drop down for any positions} 
    \begin{lemma}\label{lem:drop}
        For $\delta$ from~\autoref{lem:drop-down} and for all integer
        $k>2m$,
        \begin{equation}
            \Delta_{k+1}\leqslant \delta \overline{\Delta}_{k-{2m}}.
        \end{equation}
    \end{lemma}
    To prove~\autoref{lem:drop} we have to introduce new notation and prove an
    auxiliary proposition. Let
    \newcommand{\Delm}{\Delta^-}
    \newcommand{\Delp}{\Delta^+}
    \newcommand{\bDelm}{\overline{\Delta}^-}
    \newcommand{\bDelp}{\overline{\Delta}^+}
    \begin{align*}
        \Delm_k  =\max\left\{0, \frac{1}{2}-p_k\right\},&\quad
        \Delp_k  =\max\left\{0, p_k-\frac{1}{2}\right\},\\
        \bDelm_k =\max_{i\in W_k} \Delm_i,&\quad
        \bDelp_k =\max_{i\in W_k} \Delp_i.
    \end{align*}
    Obviously, $\overline{\Delta}_k=\max\{\bDelm_k, \bDelp_k\}$.
    \begin{proposition}\label{prop:dkpdkm}
        For any natural $k$ the following holds:
        $$\Delta_{k+1}^+ \le \bDelm_k.$$
    \end{proposition}
    \begin{proof}
        If $p_{k+1}\le 1/2$, then $\Delp_{k+1}=0\le \bDelm_{k}$ by definition of $\bDelm_k$. Consider case $p_{k+1}\ge 1/2$. Then for some $\boldpi \in K$,
        \begin{multline}
            p_{k+1}-\frac{1}{2}=\frac{1}{2}-\sum_{i=1}^m \pi_i p_{k-i+1}=
            \sum_{i=1}^m\pi_i \left(\frac{1}{2}- p_{k-i+1}\right) \\
            \le \sum_{\substack{i=1,\\ p_{k-i+1} \le 1/2}}^{m}
                    \pi_i\left(\frac{1}{2}- p_{k-i+1}\right)
            \le \sum_{\substack{i=1,\\ p_{k-i+1} \le 1/2}}^{m}
                    \pi_i \bar{\Delta}_k^- \\
            \le \sum_{i=1}^m \pi_i \bar{\Delta}_k^-=\bar{\Delta}_k^-.
        \end{multline}
    \end{proof}
    Now we can prove \autoref{lem:drop}.
    \begin{proof}[Proof of \autoref{lem:drop}]
        If $p_{k+1}<1/2$, \autoref{lem:drop-down} implies:
        $$
        \Delta_{k+1} \leqslant \delta \overline{\Delta}_{k-m} \leqslant
        \delta \overline{\Delta}_{k-2m}
        $$
        and the lemma is proved. (The last inequality is due
        to~\autoref{cor:monot}.)

        Now assume $p_{k+1}\geqslant 1/2$. In this case
        $\Delta_{k+1}=\Delta_{k+1}^+ \leqslant \bDelm_k$ due to
        \autoref{prop:dkpdkm}. For all $j \in W_k$ such that $p_j<1/2$,
        \autoref{lem:drop-down} implies:
        $$
            \Delta_j^-=\Delta_j\leqslant \delta
            \overline{\Delta}_{j-1-m}\leqslant
            \delta\overline{\Delta}_{k-2m}.
        $$
        Again, the last inequality is due to \autoref{cor:monot} since $j\geqslant
        k-m+1$. Therefore, $\bDelm_k\leqslant \delta \overline{\Delta}_{k-2m}$.
        This finishes the proof of \autoref{lem:drop}.
    \end{proof}
    \begin{corollary}\label{cor:final-drop}
        For all integer $k>3m$, $\overline{\Delta}_{k}\leqslant
        \delta \overline{\Delta}_{k-3m}$.
    \end{corollary}
    \begin{proof}
    From definition of $\overline{\Delta}_{k}$, \autoref{lem:drop} and \autoref{cor:monot} it follows that
$$
\overline{\Delta}_{k} = \max(\Delta_k,\ldots , {\Delta}_{k-m+1})\leqslant \delta\max(\overline{\Delta}_{k-2m-1},\ldots , \overline{\Delta}_{k-3m})= \delta \overline{\Delta}_{k-3m}.
$$        
    \end{proof}

    Now we are ready to finish the proof of the main result. Let $k_N=1+3mN$ for
    arbitrary integer $N$. Inductive application of \autoref{cor:final-drop}
    implies:
    $$
        \overline{\Delta}_{k_N}\leqslant
        \delta^{N}\overline{\Delta}_1=\frac{1}{2}\delta^N\to 0\text{ as }N\to \infty.
    $$
    Due to monotonicity of $\overline{\Delta}_k$, this implies:
    $$
        \lim_{k\to \infty} \overline{\Delta}_k \to 0.
    $$
    By definition of $\overline{\Delta}_k$, $\Delta_k \leqslant
    \overline{\Delta}_k$ and therefore: 
    $$
        \lim_{k\to \infty} {\Delta}_k \to 0
    $$
    which is equivalent to~\eqref{eq:main}. \autoref{thm:main} is proved modulo~\autoref{lem:corridor}.
    


\bigskip

\textit{This research did not receive any specific grant from funding agencies in the public, commercial, or not-for-profit sectors. Declarations of interest: none.}

\section*{Appendix}

    In this Appendix, we prove \autoref{lem:corridor}.
    
    \begin{proof}
        Take any $\boldpi=(\pi_1, \ldots, \pi_m) \in K$. Since the players are
        rational~\eqref{eq:pk}, 
        $$
            p_{k+1}\geqslant 1 - \sum_{i=1}^{m}\pi_ip_{k-i+1},
        $$
        or equivalently,
        $$
        \sum_{i=1}^{m}\pi_ip_{k-i+1} \geqslant 1 - p_{k+1}.
        $$
        Due to Lemma assumption, $p_{k+1}<\frac12$ and therefore
        $p_{k+1}=\frac12-\Delta_{k+1}$. We have:
        $$
            \sum_{i=1}^{m}\pi_i p_{k-i+1} \geqslant  
            1 - \left(\frac{1}{2} - \Delta_{k+1}\right) = 
            \frac{1}{2} + \Delta_{k+1}.
        $$
        Then, the following inequality holds:
        \begin{multline}
            \label{zero}
            \sum_{i=1}^{m}\pi_i\left(p_{k-i+1}-
                \left(\frac{1}{2}+\Delta_{k+1}\right)\right)=\\ 
            \sum_{i=1}^{m}\pi_i p_{k-i+1}-
            \sum_{i=1}^{m}\pi_i\left(\frac{1}{2}+\Delta_{k+1}\right)\geqslant\\
          \left(\frac{1}{2}+\Delta_{k+1}\right)-\left(\frac{1}{2}+\Delta_{k+1}\right)=0.
        \end{multline}
        Now take arbitrary
        \begin{equation}      
           j\in \argmax_{1\leqslant i\leqslant m}\left(\frac{1}{2}+\Delta_{k+1}-p_{k-i+1}\right).
           \label{jargmax}         
        \end{equation} 
        By definition of $\nu$ and Theorem assumption $\nu>0$ (see \eqref{nu}), there exists a strategy $\widehat\boldpi=(\widehat\pi_1, \ldots, \widehat\pi_m) \in K$ such that
                \begin{equation}      
                   \widehat\pi_j\geqslant \nu > 0.
                   \label{pij}         
                \end{equation}
        Inequality \eqref{zero} holds for arbitrary $\boldpi$ and therefore it
        holds for
        $\widehat\boldpi$. Rewrite it in the following way, separating the term with
        $i=j$ from the rest of the sum:
        $$
        \sum_{\substack{1\leqslant i\leqslant m \\ i\neq j}}\widehat\pi_i\left(p_{k-i+1}-\left(\frac{1}{2}+\Delta_{k+1}\right)\right)+ \widehat\pi_j\left(p_{k-j+1}-\left(\frac{1}{2}+\Delta_{k+1}\right)\right)\geqslant 0.
        $$
        Then we have the following sequence of estimates:
        \begin{multline}
                   -\widehat\pi_j\left(p_{k-j+1}-\left(\frac{1}{2}+\Delta_{k+1}\right)\right)\leqslant\sum_{\substack{1\leqslant i\leqslant m \\ i\neq j}}\widehat\pi_i\left(p_{k-i+1}-\left(\frac{1}{2}+\Delta_{k+1}\right)\right) \leqslant        
                   \\
                   \sum_{\substack{1\leqslant i\leqslant m \\ i\neq j}}\widehat\pi_i\max_{\substack{1\leqslant t\leqslant m}} \left(p_{k-t+1}-\left(\frac{1}{2}+\Delta_{k+1}\right)\right) =
                   \\
                  \left(\sum_{\substack{1\leqslant i\leqslant m \\ i\neq j}}\widehat\pi_i\right)\cdot \max_{\substack{1\leqslant t\leqslant m }} \left(p_{k-t+1}-\left(\frac{1}{2}+\Delta_{k+1}\right)\right)=
                  \\
                     (1-\widehat\pi_j)  \max_{\substack{1\leqslant t\leqslant m}} \left(p_{k-t+1}-\left(\frac{1}{2}+\Delta_{k+1}\right)\right),
        \label{pimax}                 
        \end{multline}
        where the last equality follows from the fact that 
        $$\sum_{i=1}^m\widehat\pi_i=1.$$
        From \eqref{pimax} we derive the lower estimate for the left-hand side of the Corridor lemma inequality \eqref{corridor}:
        \begin{multline}
        \max_{\substack{1\leqslant t\leqslant m}} \left(p_{k-t+1}-\left(\frac{1}{2}+\Delta_{k+1}\right)\right)\geqslant -\frac{\widehat\pi_j}{1-\widehat\pi_j} \left(p_{k-j+1}-\left(\frac{1}{2}+\Delta_{k+1}\right)\right)=
        \\
        \frac{\widehat\pi_j}{1-\widehat\pi_j} \left(\frac{1}{2}+\Delta_{k+1}-p_{k-j+1}\right).
        \label{almost}
        \end{multline} 
        Note that 
        $$
        \frac{1}{2}+\Delta_{k+1}-p_{k-j+1}\geqslant 0.
        $$
        Indeed, otherwise, from the definition of $j$ (see \eqref{jargmax}) it
        would follow that for all $i=1,\ldots, m$,
        $$
        \frac{1}{2}+\Delta_{k+1}-p_{k-i+1}< 0
        $$
        or
        $$
        p_{k-i+1}> \frac{1}{2}+\Delta_{k+1}.
        $$ 
        However, this is impossible because for optimal strategy
        $(\pi_1,\ldots,\pi_m)$ we have:
        $$p_{k+1}=1-\sum_{i=1}^m p_{k-i+1} \pi_i <1 -\sum_{i=1}^m \pi_i \left(\frac{1}{2}+\Delta_{k+1}\right) = \frac{1}{2}-\Delta_{k+1}$$ 
        whereas $p_{k+1}=\frac{1}{2}-\Delta_{k+1}$ by definition.

        Note that function $x\mapsto \frac{x}{1-x}$ is increasing for $x \in (0,1)$. Thus
        we can estimate $\frac{\widehat \pi_j}{1-\widehat \pi_j}$ by
        $\frac{\nu}{1-\nu}$ from below in~\eqref{almost} and obtain 
        \begin{multline}
        \max_{\substack{1\leqslant t\leqslant m}} \left(p_{k-t+1}-\left(\frac{1}{2}+\Delta_{k+1}\right)\right)\geqslant \frac{\nu}{1-\nu} \left(\left(\frac{1}{2}+\Delta_{k+1}\right)-p_{k-j+1} \right)=
        \\
        \frac{\nu}{1-\nu}\max_{\substack{1\leqslant t\leqslant m}}\left(\frac{1}{2}+\Delta_{k+1}-p_{k-t+1}\right).
        \end{multline}
        The last equality follows from the definition of $j$ (see \eqref{jargmax}). This finishes the proof
        of \autoref{lem:corridor} and the main result (\autoref{thm:main}).
    \end{proof}

\end{document}